\documentclass{amsart}

\usepackage{ifthen}

\newtheorem{anyprop}{Anyprop}[section]

\newtheorem{theorem}[anyprop]{Theorem}
\newtheorem{lemma}[anyprop]{Lemma}
\newtheorem{proposition}[anyprop]{Proposition}
\newtheorem{corollary}[anyprop]{Corollary}

\theoremstyle{definition}

\newtheorem{definition}[anyprop]{Definition}

\newtheorem{example}[anyprop]{Example}

\newtheorem{remark}[anyprop]{Remark}

\newtheorem{definitionintro}{Definition}

\newcommand  {\fop}     {\mathfrak{p}}

\newcommand  {\foq}     {\mathfrak{q}}

\newcommand  {\Spec}    {\operatorname{Spec}}

\theoremstyle{remark}

\numberwithin{equation}{section}

\usepackage{amscd}
\usepackage{amssymb}
\input xy
\xyoption{all}

\newcommand{\idealq}{{\mathfrak q}}

\newcommand{\idealp}{{\mathfrak p}}

\newcommand{\idealm}{{\mathfrak m}}

\newcommand{\pro}{\varphi}

\begin{document}
\title[ON THE CONNECTEDNESS OF THE SPECTRUM OF FORCING ALGEBRAS]
{ON THE CONNECTEDNESS OF THE SPECTRUM OF FORCING ALGEBRAS}

\author[Holger Brenner]{Holger Brenner}
\author[Danny de Jes\'us G\'omez-Ram\'irez]{Danny de Jes\'us G\'omez-Ram\'irez}

\address{Universit\"at Osnabr\"uck, Fachbereich 6: Mathematik/Informatik,
Albrechtstr. 28a,
49069 Osnabr\"uck, Germany}
\address{Universidad Nacional de Colombia, sede Medell\'in, Escuela de Matem\'aticas, Calle 59 A No 63-20 Oficina 43-106
Medell\'in, Colombia}

\email{hbrenner@uos.de}
\email{dgomezramire@uos.de}




\begin{abstract}
We study the connectedness property of the spectrum of forcing algebras over a noetherian ring.
In particular we present for an integral base ring a geometric criterion for connectedness in terms of horizontal and vertical components of the forcing algebra.
This criterion allows further simplifications when the base ring is local, or one-dimensional, or factorial.
Besides, we discuss whether the connectedness is a local property. Finally, we present a characterization of the integral closure of an ideal by means of the universal connectedness
of the forcing algebra.
\end{abstract}

\maketitle

\smallskip


\section*{Introduction}
Let $R$ be a commutative ring, $I=(f_1,\ldots ,f_n)$ a finitely generated ideal and $f$ an arbitrary element of
$R$. A very natural and important question, not only from the theoretical but also from the computational point
of view, is to determine if $f$ belongs to the ideal $I$ or to some ideal closure of it
(for example to the radical, the integral closure, the solid closure, the tight closure, among others).
To answer this question the concept of a forcing algebra introducted by Mel Hochster in the context of
solid closure \cite{hochstersolid} is important (for more information on forcing algebras see \cite{brennerforcingalgebra},
\cite{brennerbengaluru} and \cite{gomezthesis}):

\begin{definitionintro} 
Let $R$ be a commutative ring, $I=(f_1, \ldots , f_n)$ an ideal and $f \in R$ another element. Then
the \emph{forcing algebra} of these (forcing) data is 
\[  A=R[T_1, \ldots, T_n] /(f_1T_1 + \ldots + f_nT_n+f ) \,  . \]
\end{definitionintro}

Intuitively, when we divide by the forcing equation $f_1T_1 + \ldots + f_nT_n+f $ we are ``forcing'' the element
$f$ to belong to the expansion of $I$ in $A$. Besides, it has the universal property that for any $R$-algebra $S$
such that $f\in IS$, there exists a (non-unique) homomorphism of $R$-algebras $\theta: A\rightarrow S$.
Furthermore, the formation of forcing algebras commutes with arbitrary change of base. Formally, if
$\alpha: R\rightarrow S$ is a homomorphism of rings, then $S\otimes_R A$ is the forcing algebra for the
forcing data $\alpha(f_1),\ldots ,\alpha(f_n),\alpha(f)$ \cite[chapter 2]{brennerbarcelona}. In particular,
if $\idealp \in X =\Spec R$, then the fiber of (the forcing morphism) $\varphi: \Spec A \rightarrow \Spec R$ over $\idealp $, $\varphi^{-1}( \idealp)$, is the scheme theoretical
fiber $\Spec (\kappa (\idealp)\otimes_R A)$, where $\kappa (\idealp)=R_\idealp/ \idealp R_\idealp$ is its residue field.
In this case, the fiber ring $\kappa (\idealp) \otimes_R A$ is the forcing algebra over $\kappa (\idealp)$
corresponding to the forcing data $f_1(\idealp),\ldots ,f_n(\idealp),f(\idealp)$, where we
denote by $g(\idealp)\in \kappa (\idealp)$, the image (the evaluation) of $g\in R$ inside the residue field $\kappa (\idealp)$. 

One can also define forcing algebras by several forcing equations and write them in a matrix form:
\[A=R[T_1,\ldots ,T_n]/\left\langle \left(\begin{array}{ccc} f_{11}& \ldots &f_{1n} \\ \vdots& \ddots & \vdots \\ f_{m1} & \ldots & f_{mn} \end{array}  \right) \cdot \left( \begin{array}{c}T_1 \\ \vdots \\ T_n \end{array}\right)+\left( \begin{array}{c}f_1\\ \vdots \\ f_m\end{array}
\right)   \right\rangle. \]

This corresponds to a submodule $N \subseteq M $ of finitely generated $R$-modules and an element $f \in M$ via a
free representation of these data (see \cite{brennerforcingalgebra}). The fiber over a point $\idealp \in \Spec R$ of this forcing algebra is just the solution set of the
corresponding system of inhomogeneous linear equations over $\kappa (\idealp)$. If the  vector $(f_1, \ldots , f_m)$
is zero, then we are dealing with a ``homogeneous'' forcing algebra. In this case there is a
(zero- or ``horizontal'') section $s: X =\Spec R \rightarrow Y = \Spec A$ coming from the homomorphism of
$R$-algebras from $A$ to $R$ sending each $T_i$ to zero. This section sends a prime ideal $\idealp \in X$ to the
prime ideal $(T_1,\ldots ,T_n)+\idealp \in Y$.

In this paper we discuss when (the spectrum of) a forcing algebra is connected. The fibers of a forcing algebra are affine spaces hence connected
unless they are empty (Lemma \ref{forcingfiber}). In the homogeneous case it is easy to show that if
$X=\Spec R$ is a connected topological space, so is $Y=\Spec A$ (Proposition \ref{connectedprop}(4)).
However, in the non-homogeneous case this question is more subtle. We establish in the noetherian case
(in the setting of an ideal, i.e. $m=1$) a condition equivalent to the connectedness of $Y$
in terms of the horizontal and the vertical (to be defined) irreducible components of $\Spec A$ (Theorem \ref{connectedcomponents}). In the following we specialize then to the cases where $R$ is one-dimensional (Corollary \ref{connectedcomponentsone}),
local (Corollary \ref{connectedcomponentslocal}), a factorial domain
(Corollary \ref{factorialconnected}) and a principal ideal domain (Corollary \ref{PID}).
In the fifth section we study the problem whether the connectedness of a forcing algebra is a local
property over the base. We give a quite general result (Theorem \ref{theoremlocal}) that local connectedness (over the base) implies connectedness.
In the one-dimensonal domain case the converse is true for forcing algebras (Corollary \ref{onedimlocal}), but
this can neither be extended to higher dimension nor to one-dimensional rings which are not domains.
In the final section we relate the integral closure of an ideal to the universal connectedness of the forcing algebra.
 
This paper arose during the research stay of Danny de Jes\'us G\'omez-Ram\'irez at the University of Osnabr\"uck in 2011.
He thanks all the department there for the help and support, in particular, to Daniel Brinkmann for inspiring
discussions on the subject. He also thanks the Universidad Nacional de Colombia and the German Academic Exchange
Service (DAAD) for financial support. Finally, we want to thank the referee for careful reading and many suggestions.

\section{Connectedness}

Recall that a topological space $X$ is connected if there exists exactly two subsets (namely
$\emptyset$ and $X \neq \emptyset$)
which are open as well as closed. A connected component of $X$ is a maximal connected subspace, i.e., 
not strictly contained in any connected subspace of $X$. Every connected component is necesarilly
closed because its closure is a closed connected set containing it. Moreover, the connected components
form a partition of the space $X$. 

For a commutative ring $A$, the spectrum $\Spec A$ is connected if and only if $A \neq 0$ and if it is not possible to write
$A= A_1 \times A_2$ with $A_1,A_2 \neq 0$. Equivalently, there exist exactly
two idempotent elements,
namely $0$ and $1$ (see for example \cite[Exercise 2.19, Chapter II]{hartshornealgebraic} or \cite[Exercise 2.25]{eisenbud}). Hence domains and
local rings are connected. If $A$ is a $\mathbb C$-algebra of finite type, then the connectedness of $\Spec A$ (in the Zariski topology) is equivalent to the connectedness of the complex space $(\Spec A)_{\mathbb C}$ in the complex topology, because irreducible varieties are connected over $\mathbb C$ (the real spectrum might be disconnected though).

Now, we present a lemma describing the fibers of a forcing algebra as affine spaces over the base residue field.

\begin{lemma}
\label{forcingfiber}
 Let $R$ be a commutative ring and let $A$ be the forcing algebra corresponding to the data
$\left\lbrace f_{ij}, f_i \right \rbrace$. Let $\idealp \in X$ be an arbitrary prime ideal of $R$ and $r$ the rank of the
matrix $\left\lbrace f_{ij}(\idealp) \right\rbrace$. Then the fiber over $\idealp $ is empty or isomorphic to the
affine space ${ \mathbb A}^{n-r}_{\kappa (\idealp)}$.
\end{lemma}
\begin{proof}
We know by a previous remark that the fiber ring over $P$ is $\kappa (\idealp)\otimes_RA$ which is just 
\[\kappa (\idealp)[T_1,\ldots ,T_n]/\left\langle \left(\begin{array}{ccc} f_{11}(\idealp)&\ldots &f_{1n}(\idealp)\\ \vdots& \ddots &\vdots \\f_{m1}(\idealp)&\ldots &f_{mn}(\idealp)\end{array}  \right) \cdot \left( \begin{array}{c}T_1 \\ \vdots \\ T_n \end{array}\right)+\left( \begin{array}{c}f_1(\idealp)\\ \vdots \\ f_m(\idealp)\end{array}
\right)  \right\rangle \, .\]
Hence we are dealing with a forcing algebra over a field. With the usual techniques of linear algebra
(elementary row operations and exchange of columns) we can bring
the matrix to triangular form without changing (up to isomorphism) the algebra. The linear system has no solution if and only if the
forcing algebra is $0$. Otherwise we can use $r$ rows to eliminate $r$ variables showing that the algebra is
isomorphic to a polynomial algebra in $n-r$ variables.
\end{proof}

Since an affine space is irreducible and hence connected, the preceeding lemma tells us that the fibers of a forcing algebra are connected unless they are empty. The easiest example of an empty forcing algebra is $K[T]/(0T-1)$, where $K$ is a field. A forcing algebra may be connected though some fibers may be empty, an example is given by $K[X,Y][T_1,T_2]/(XT_1+YT_2-1)$.

In the following we will mainly deal with the case where all fibers are non-empty. This is equivalent to say that $f$ belongs to the radical
of the ideal $I$ (or, by definition, to the radical of the submodule $N$, see \cite{brennerforcingalgebra}).

\begin{proposition}
\label{connectedprop}
Let $A$ be a forcing algeba over $R$ with the corresponding morphism $\pro:Y= \Spec A \rightarrow X  = \Spec R $. Then the
following hold.
\begin{enumerate}

\item
The connected components of $Y $ are of the form $\pro^{-1} (Z)$ for suitable subsets
$Z \subseteq X $.

\item
If $\pro : Y \rightarrow X$ is surjective, then these $Z$ are uniquely determined.

\item
If  $\pro : Y \rightarrow X$ is surjective and $Y$ is connected, then $X$ is connected.

\item
If the forcing data are homogeneous, then there is a bijection between the connected components of $X$ and $Y$.
In particular, $X$ is connected if and only if $Y$ is connected.

\item
Suppose that $\pro : Y \rightarrow X$ is (surjective and) a submersion. Then there is a bijection between the
connected components of $X$ and $Y$. In particular, $X$ is connected if and only if $Y$ is connected.
\end{enumerate}
\end{proposition}
\begin{proof}
(1). By Lemma \ref{forcingfiber} the fibers are connected. Hence a connected component of $Y$ which contains a
point of a fiber contains already the complete fiber. (2) and (3) are clear.

(4). First note that in the homogeneous case there exists an $R$-algebra-homomorphism $A \rightarrow R$ given by
$T_i \mapsto 0$, which induces a section $s:X \rightarrow Y$ given by $\idealp \mapsto (\idealp,T_1, \ldots, T_n)$.
Therefore  $\pro : Y \rightarrow X$ is surjective. Suppose first that $X$ is connected. Then the image $s(X)$ of the section is connected.
Hence two points $Q_1,Q_2 \in Y$ lie in the connected subset $\pro^{-1} ( \pro(Q_1)) \cup s(X) \cup  \pro^{-1} ( \pro(Q_2)) $ and so $Y$ is connected. This argument holds for every subset $Z \subseteq X$, so the statement about the components follows.

(5). Recall that a submersion $\pro:Y \rightarrow X$ between topological spaces is surjective and has the property
that $\pro^{-1}(T)$ is open if and only if $T \subseteq X$ is open. So if $\pro^{-1}(Z)$ is open and closed, then $Z$
itself is open and closed. Hence $Y$ is connected if and only if $X$ is connected. The statement about the
components follows.
\end{proof}

\begin{example}
The conditions in Proposition \ref{connectedprop} (4), (5) are necessary, as the following example shows.
Consider $R=K[X]$, where $K$ is a field, and the nonhomogeneous forcing algebra $A=R[T]/(X^2T-X)$. The minimal primes of $(X^2T-X)$ are $(X)$ and
$(XT-1)$, which are comaximal (since $1= XT -(XT-1)$). So by the Chinese Remainder theorem
$A \cong R[T]/(X)\times R[T]/(X T-1)$ and therefore
\[ \Spec A= \Spec k[T] \uplus \Spec K[X,T]/(X T-1)  \, ,\]
i.e. a disjoint union of a line over a point and a hyperbola dominating the base (its image is the pointed affine line, hence dense). But, $\Spec R$ is the affine
line which is connected. Note that the element $X$ belongs to the radical of $(X^2)$, but does not belong to $(X^2)$ nor to its integral closure.
Hence $\pro: \Spec A \rightarrow \Spec R$ is surjective, but not submersive (see also Section \ref{integral}).
\end{example}

\section{Horizontal and vertical components}

We describe now the irreducible components of the spectrum of a forcing algebra over an integral base in the ideal case. We will identify prime ideals inside $R[T_1, \ldots, T_n]$ minimal over $(f_1T_1+ \ldots + f_nT_n+f)$ with the minimal prime ideals of the forcing algebra  $R[T_1, \ldots, T_n]/(f_1T_1+ \ldots + f_nT_n+f)$.
 
\begin{lemma}
\label{horizontalvertical}
Let $R$ be a noetherian domain, $I=(f_1 , \ldots ,f_n) $ an ideal, $f \in R$ and
$A=R[T_1, \ldots , T_n]/(f_1T_1+ \ldots +f_nT_n+f)$
the forcing algebra for these data. Then the following hold.

\begin{enumerate}

\item
For $I \neq 0$ there exists a unique irreducible component $H \subseteq \Spec A$ with the property of dominating the
base $\Spec R$ (i.e. the image of $H$ is dense).
This component is given (inside $R[ T_1, \ldots , T_n]$) by
\[\idealp = R[ T_1, \ldots , T_n] \cap (f_1T_1+ \ldots + f_nT_n+f) Q(R)[ T_1, \ldots , T_n] \, ,\]
where $Q(R)$ denotes the quotient field of $R$.

\item
All other irreducible components of $\Spec A$ are of the form $V( \idealq  R[ T_1, \ldots , T_n] )$ for
some prime ideal $\idealq \subseteq R$ which is minimal over $(f_1, \ldots , f_n,f)$.

\item
For a minimal prime ideal  $\idealq \subseteq R$ over $(f_1, \ldots , f_n,f)$ the extended ideal
$\idealq   R[ T_1, \ldots , T_n]$ defines a component of $\Spec A$
if and only if $I=0$ or $I \neq 0$ and there exists a polynomial $G\in \idealp $,
$G \not\in \idealq   R[ T_1, \ldots , T_n]  $.
\end{enumerate}
\end{lemma}
\begin{proof}
(i). For $I \neq 0$ the polynomial $ f_1T_1+ \ldots + f_nT_n+f $ is irreducible (thus prime) in
$Q(R) [ T_1, \ldots , T_n] $, hence
the intersection of this principal ideal with $R[ T_1, \ldots , T_n]$ gives a prime ideal in this polynomial
ring and therefore in $R [T_1, \ldots , T_n]/(f_1T_1+ \ldots + f_nT_n+f)$. The minimality is clear, since it holds in
a localization.
Because of $R \cap \idealp= 0$, this component dominates the base. On the other hand, let $\idealp'$ be a minimal
prime ideal that is
minimal over $(f_1T_1 + \ldots +f_nT_n+f)$ and suppose that $R \cap \idealp'= 0$.
Let $h \in \idealp$.
Then there exists $r,s \in R$, $r \neq 0$, and a polynomial $G \in R[T_1, \ldots, T_n] $ such that
$r h=s G (f_1T_1+ \ldots + f_nT_n+f)$. This element belongs to $\idealp'$ and since $r \not\in \idealp'$
we deduce $h \in \idealp'$.
Hence $\idealp'=\idealp$.

(ii). Let $(f_1 T_1 + \ldots + f_nT_n+f) \subseteq \idealq$ be a minimal prime ideal different from $\idealp$.
In every localization $A_{f_i}$ there is only one minimal prime ideal, namely the horizontal component (as in (1)), since
 $f_1 T_1 + \ldots +f_nT_n+f$ is a prime element whenever some $f_i$ is a unit. Therefore
$\idealq \not\in \Spec A_{f_i}$
and hence $f_i \in \idealq$. Because $\idealq$ contains the forcing equation we also deduce $f \in \idealq$.
Hence $(f_1, \ldots ,f_n,f) \subseteq \idealq$
and by the minimality condition $\idealq$ is minimal over the extended ideal $(f_1, \ldots ,f_n,f)R[T_1, \ldots, T_n] $. Therefore $\idealq$ must be the extended ideal of a
minimal prime ideal of $(f_1, \ldots ,f_n,f)$ in $R$ (the minimal prime ideals above $(f_1, \ldots,f_n,f)R$, above
 $(f_1, \ldots,f_n,f) R[T_1, \ldots, T_n]$ and above  $(f_1, \ldots,f_n,f)A$ are in bijection).

(iii). Let $\idealq$ be a minimal prime ideal of $(f_1, \ldots ,f_n,f)$ in $R$. Then $\idealq A$ is a
minimal prime ideal of $\Spec A$ if and only if $\idealp \not \subseteq \idealq A$
(since by (ii) we know there are no other possible minimal prime ideals). 
\end{proof}

We call $H=V(\idealp)$ the \emph{horizontal component} of the forcing algebra and the other components $V(\idealq_j)$
the \emph{vertical components}. If $I=0$ there exist only the vertical components which are in bijection with the
components of $ \Spec R/(f)$. If $n=1$, $R$ is a domain and $f_1 \neq 0$, then the horizontal component is just the closure of the
graph of the rational function
$T=- \frac{f}{f_1}$ inside $\Spec R \times {\mathbb A}^1$.

\begin{remark}
If $R$ is a noetherian domain and $I=(f_1 , \ldots, f_n) \neq 0$, then the horizontal component exists
and the describing prime ideal $\idealp$ has height one in $R[T_1, \ldots, T_n]$. If $\idealq$
is a minimal prime ideal over $(f_1 , \ldots, f_n,f)$ of height one, then the extended ideal
in the polynomial ring has also height one and can therefore not contain the horizontal prime ideal.
Therefore such prime ideals yield vertical components. 
\end{remark}

It is possible that all the $V( \idealq_j)$, where $\idealq_j \supseteq (f_1, \ldots, f_n,f)$ is a minimal prime ideal, lie on the horizontal component.
In this case there exists no vertical component.

\begin{example}
Let $K$ be a field, $R=K[X,Y]$ and consider the forcing algebra $A=R[T]/(XT+Y)$, which is a domain.
Here the only candidate for a vertical component, namely $V(X,Y)$, is not a component of $\Spec A$, because it
lies on the horizontal component.
\end{example}

If the forcing equation has a nice factorization inside the polynomial ring $R[T_1, \ldots, T_n]$,
then we can describe the horizontal and vertical components more explicitly.

\begin{lemma}
\label{forcingprimedecomposition}
Let $R$ be an noetherian integral domain and $A$ be a forcing algebra over $R$ with forcing equation $h=f_1T_1+ \cdots + f_nT_n +f=dh'$,
where $(f_1, \ldots, f_n) \neq 0$, $d\in R$ and where $h'=f_1'T_1+\cdots+f_n'T_n+f'$ is a prime element in $B=R[T_1, \ldots, T_n]$.
Then the horizontal component of $\Spec A$ is $V(h')$ and the vertical components of $\Spec A$ are $V(\foq A)$, where $\foq$ varies over
the minimal prime ideals of $(d)$ in $R$. 
\end{lemma}
\begin{proof}
Because $h'$ is a prime element we have
 \[(h')R[T_1, \ldots, T_n]  =R[T_1, \ldots, T_n] \cap (h) Q(R)[T_1, \ldots, T_n]  \]
and this gives by Lemma \ref{horizontalvertical} (1) the horizontal component. By Lemma \ref{horizontalvertical} (3) we have to show that the
minimal prime ideals over $(d)$ correspond to the minimal prime ideals over
$J=(f_1, \ldots, f_n,f )$ with the additional property that their extension to $R[T_1, \ldots, T_n]$ does not contain $h'$.

So let $\foq$ be a minimal prime ideal over $(d)$. Then by Krull's theorem the height of $\foq$ is $1$
and hence it is also minimal over $J \subseteq (d)$. Moreover, the height of $\foq R[T_1, \ldots, T_n]$ is also $1$. Besides, $h' \not\in \foq b$,
for otherwise $(0)\subsetneq(h')\subsetneq \foq R[T_1, \ldots, T_n]$ would be a chain of prime ideals of length $2$,
because $(h')\neq\foq B$, since $(h') \cap R=0$ and $\foq R[T_1, \ldots, T_n]  \cap R= \foq$.

Conversely, let $\foq$ denote a minimal prime ideal of $J$ such that $h'$ does not belong to $ \foq R[T_1, \ldots , T_n]$. Assume that $d \not\in \foq$.
Then because of $f_1'd, \ldots , f_nd,f'd \in \foq$ we get immediately $f_1, \ldots , f_n,f \in \foq$ and hence the contradiction $h'\in \foq R[T_1, \ldots, T_n]$.
So we must have $d \in \foq$
and $\foq$ must also be minimal over $(d)$.
\end{proof}

If $R$ is a (not necessarily noetherian) factorial domain (a unique factorization domain), then there exists always a factorization $h=dh'$ with $h'$ a prime element in $B$.
The minimal prime ideals over $(d)$ are given by the prime factors $p$ of $d$, and $pB$ has also height $1$. Hence the argument of the lemma goes
through also in this case. Example \ref{behnkekegelexample} below shows that a forcing equation does not always have a prime decomposition as in the lemma. Then it is more complicated to determine the vertical components.

The following example shows that the irreducible components in the module case are more complicated.
In particular, Lemma \ref{horizontalvertical} (2) is not true.

\begin{example}
Consider over $R=K[X,Y]$ the forcing algebra
\[A= R[T_1,T_2]/( XT_1-XY, YT_2-XY  ) \cong R[T_1,T_2]/ \! \left(\!\! \begin{pmatrix} X &0 \\ 0 &Y \end{pmatrix} \begin{pmatrix}  T_1  \\  T_2 \end{pmatrix} - \begin{pmatrix} XY \\ XY \end{pmatrix} \!\!  \right) \!  . \]
The horizontal component of $\Spec A$ is given by the prime ideal $(T_1-Y,T_2-X)$. The algebra is connected, since this component
is a section. The other minimal prime ideals are $(X,T_2)$, $(Y,T_1)$ and $(X,Y)$. Only the last one is the extension of a prime ideal of the base.
\end{example}

\section{Connectedness results}

The following is our main general connectedness result on forcing algebras.

\begin{theorem}
\label{connectedcomponents}
Let $R$ be a noetherian domain, $I=(f_1 , \ldots ,f_n) $ an ideal $\neq 0$, $f \in R$ and
$A=R[T_1, \ldots , T_n]/(f_1T_1+ \ldots +f_nT_n+f)$
the forcing algebra for these data. Let $H=V(\idealp)$ be the horizontal component of $\Spec A$ and let
$V_j=V(\idealq_j)$, $j \in J$, be the vertical components of $\Spec A$ according to Lemma \ref{horizontalvertical}.
Let $Z_i= \bigcup_{j \in J_i} V_j$ be the connected components of $\bigcup_{j \in J} V(\idealq_j)$. 
Then $\Spec A$ is connected if and only if $H$ intersects every $Z_i$.
\end{theorem}
\begin{proof}
This is clear from Lemma \ref{horizontalvertical} since the connected components of a noetherian scheme are just the unions of its irreducible
components which intersect chainwise. Note also that the $Z_i$ can be determined in $\Spec R$ alone.
\end{proof}

\begin{corollary}
\label{connectedcomponentsone}
Let $R$ be a noetherian domain of dimension $1$, $I=(f_1 , \ldots ,f_n) $ an ideal $\neq 0$, $f \in R$ and
$A=R[T_1, \ldots , T_n]/(f_1T_1+ \ldots +f_nT_n+f)$
the forcing algebra for these data. Let $H=V(\idealp)$ be the horizontal component of $\Spec A$ and let
$V_j=V(\idealq_j)$, $j \in J$, be the vertical components of $\Spec A$ according to Lemma \ref{horizontalvertical}.
Then $\Spec A$ is connected if and only if $H$ intersects every $V_j$.
\end{corollary}
\begin{proof}
This follows from Theorem \ref{connectedcomponents} since the minimal prime ideals of $I \neq 0$ in a
one-dimensional domain are maximal ideals. These maximal ideals form the connected components of $V(I)$.
\end{proof}

Note that this corollary is not true in higher dimension, see Example \ref{connectedexample} in the next section. We specialize now to the local case.

\begin{corollary}
\label{connectedcomponentslocal}
Let $(R, \idealm)$ be a local noetherian domain, $I=(f_1 , \ldots ,f_n) \subseteq \idealm $ an ideal $\neq 0$, $f \in \idealm $ and
$A=R[T_1, \ldots , T_n]/(f_1T_1+ \ldots +f_nT_n+f)$
the forcing algebra for these data. Let $H=V(\idealp)$ be the horizontal component of $\Spec A$ according to Lemma \ref{horizontalvertical}.
Then $\Spec A$ is connected if and only if $\idealp + (I,f) \neq (1) $ in $R[T_1, \ldots, T_n]$.
\end{corollary}
\begin{proof}
Let $\idealq_j$, $j \in J$, be the minimal prime ideals of $(I,f)$, disregarding whether the $V_j=V(\foq_j R[T_1, \ldots ,T_n])$
give rise to vertical components of $\Spec A$ or not. Note that $V_j \cap V_i \neq \emptyset$ for all $i,j$, because we are over a local ring.
Suppose first that for at least one $j$ we have $V_j=V( \idealq_j) \subseteq H$. Then also $V_i \cap H \neq \emptyset$ and hence the forcing algebra is connected by Theorem \ref{connectedcomponents}. But this assumption also means that $ \idealp \subseteq \idealq_j R[T_1, \ldots, T_n] \subseteq \idealm R[T_1, \ldots, T_n]$ and therefore $\idealp + (I,f) \neq (1)$, because $\idealm$ does not extend to the unit ideal. So under this assumption the two properties are equivalent.

So suppose next that $V_j \not\subseteq H$ for all $j$, meaning that all $V_j$ are vertical components of $\Spec A$. The subsets $V(\idealq_j)$ inside the local ring  $\Spec R$ are connected, hence there is exactly one $Z$ in the notation of 
Theorem \ref{connectedcomponents}. By this theorem, $\Spec A$ is connected if and only if $H \cap Z \neq \emptyset $.
Because of $Z=V (\bigcap_j \idealq_j)$ this is equivalent to $ \idealp + (I,f) \neq (1)$.
\end{proof}

\begin{example}
Let $K$ be a field and set $R=K[X,Y]_{(X,Y)}$, let $A=R[T]/(XYT-X)$. The horizontal component is $V(YT-1)$ and the only vertical component is $V(X)$. Because they intersect (or because $(YT-1, X) \neq (1)$) the forcing algebra is connected. However, we have $(YT-1, \idealm)=(1)$, so the connectedness over a local ring does not imply that the horizontal component meets the fiber over the maximal ideal.
\end{example}

\section{Connectedness over factorial domains}

We deal now with the case where $R$ is a factorial domain. Note that if $R$
is factorial, then $B=R[T_1, \ldots, T_n]$ is factorial as well. So if $h=f_1T_1+\ldots +f_nT_n+f \in B$ is a forcing
equation, then one can factor out a greatest common divisor of all the coefficients $f_1, \ldots ,f_n$ and $f$,
say $d$, and obtain a representation of $h$ as a product of an element $d$ in $R$ and an irreducible polynomial
$h'= (f_1/d)T_1+ \ldots +(f_n/d)T_n+(f/d)$ in $B$ (for $n \geq 1$), which generates a prime ideal because $B$ is a factorial domain.
This hypothesis appeared already in Lemma \ref{forcingprimedecomposition} and is also crucial in the following sufficient condition for connectedness.

\begin{corollary}
\label{Zusam}
Let $R$ be a noetherian domain, $B:=R[T_1,\ldots ,T_n]$, and let (assume $(f_1, \ldots, f_n) \neq 0$)
\[h := f_1T_1+\ldots +f_nT_n+f= d(f'_1T_1+\ldots +f'_nT_n+f') \]
be a forcing equation such that  $h':=f'_1T_1+\ldots +f'_nT_n+f'$ is a prime polynomial.
Suppose that $(f'_1,\ldots ,f'_n)$ is not contained in any minimal prime ideal of $(d)$. Then $\Spec A$ is
connected, where $A=B/(h)$.
\end{corollary}
\begin{proof}
By Lemma \ref{forcingprimedecomposition}, the horizontal component of $\Spec A$ is $V(h')$ and the vertical
components correspond to the minimal prime ideals
$\foq$ over $(d)$. We will show that $V(\idealq)$ intersects the horizontal component. 

This can be established after the base
change $R \rightarrow \kappa(\idealq)$. Now at least one of the $f_i'$ becomes a unit in $\kappa(\idealq)$ and
therefore $h'$ is not a unit over $\kappa(\idealq)$. So $\kappa(\idealq)[T_1, \ldots, T_n]/(h') \neq 0$.
\end{proof}

Note also that if $d$ is a unit, then $h=h'$ is a prime polynomial by assumption and then the forcing algebra is integral,
hence connected anyway.

Now, we shall deduce a Corollary in the case that $R$ is a factorial domain. In this kind of rings we can define a greatest
common divisor of a finite set of elements $a_1,\ldots ,a_m$, denoted by $\operatorname{gcd}(a_1,\ldots ,a_m)$, using the prime
factorization, and it is well defined up to a unit in $R$ and defined as the unity in $R$ in the case that the
elements have no irreducible common factor.

\begin{lemma}
\label{factorialhorizontal}
Let $R$ be a factorial domain, $f_1, \ldots , f_n, f \in R$ with some $f_i \neq 0$ and let $d$
be a greatest common divisor of $f_1,\ldots ,f_n$ and $f$ and write
\[ h= f_1T_1+\ldots +f_nT_n+f =  d  (f'_1T_1+\ldots +f'_nT_n+f') \]
where
$ f'_i = f_i/d $. Then $h' = f'_1T_1+\ldots +f'_nT_n+f'$ is an irreducible polynomial and describes the horizontal component of $\Spec R[T_1, \ldots, T_n]/(h)$.
\end{lemma}
\begin{proof}
Suppose we have a factorization $h'=h_1h_2$ in $B=R[T_1, \ldots, T_n]$. Then one of the
$h_i$ can not contain any variable $T_j$, say $h_1$, thus $h_1\in R$. Therefore, $h_1$ divides each
$f_1/d,\ldots ,f_n/d,f/d$ and therefore it is a unit in $R$, because these elements have no
common irreducible factors. Thus $h'$ is an irreducible polynomial and hence a prime element since the polynomial ring
$B$ over $R$ is also factorial. Therefore $(h')$ describes the horizontal component by Lemma \ref{forcingprimedecomposition}. 
\end{proof}

\begin{corollary}
\label{factorialconnected}
Let $R$ be a factorial domain, $f_1, \ldots , f_n,f \in R$ with some $f_i \neq 0$ and let $d$
be a greatest common divisor of $f_1,\ldots ,f_n$ and $f$. Let
\[ h=f_1T_1+\ldots +f_nT_n+f =d ( f'_1T_1+\ldots +f'_nT_n+f') \]
be the forcing equation and let $d=p_1 \cdots p_k$ be a prime factorization of $d$. Suppose that $\{1, \ldots, k\}= \biguplus_{i \in I} J_i$ such that the $\bigcup_{j \in J_i} V(p_j)$ are the connected components of $V(d)$.
Then $\Spec B/(h)$ is connected if and only if for every $i$ there exists some $j \in J_i$ such that $(  f'_1T_1+\ldots +f'_nT_n+f', p_j) \neq (1)$.
\end{corollary}
\begin{proof}
By Lemma \ref{forcingprimedecomposition} and Lemma \ref{factorialhorizontal}, $\idealp=( f'_1T_1+\ldots +f'_nT_n+f' )$ describes
the horizontal component and every vertical component corresponds to a prime factor of $d$.
Hence the statement follows directly from Theorem \ref{connectedcomponents}.  
\end{proof}

The condition that $( f'_1T_1+\ldots +f'_nT_n+f' , p) \neq (1)$ is true if $p$ does not
divide all $f_1', \ldots, f_n'$ or if $f'$ is not a unit modulo $p$.

\begin{example}
Let $R=K[X,Y]$ over a field $K$. For $h=X(XT+Y)$ we have $d=X=p$. Because $Y$ is not a unit modulo $X$, the forcing algebra is connected.
For $h=X(XT+X+1)$ we also have $d=X=p$. Now $X+1$ is a unit modulo $X$ and the forcing algebra is not connected. For $h=XY(XT_1+YT_2+f')$ we have
$d=XY$, but $f'_1=X$, $f'_2=Y$ do not have a common prime factor and hence the forcing algebra is connected.
\end{example}

\begin{corollary}
\label{PID}
Let $R$ be a factorial domain, $B=R[T_1,\ldots ,T_n]$, $h=f_1T_1+\ldots +f_nT_n+f$ a forcing equation and $A=B/(h)$. Let $d$
be a greatest common divisor of $f_1,\ldots ,f_n$ and $f$, and assume that at least one of the $f_i \neq 0$.
Suppose that $\operatorname{gcd}(d,f_1/d,\ldots ,f_n/d)=1$, then $\Spec A$ is connected.
Moreover, if $R$ is a principal ideal domain then the condition $\operatorname{gcd}(d,f_1/d,\ldots ,f_n/d)=1$ is equivalent to $\Spec A$ being connected.
\end{corollary}
\begin{proof}
We write $h=d h'$, where $h'= f_1' T_1+\ldots +f_n' T_n+f' $ where $f_i'=f_i/d$ and $f'=f/d$.
Let $p$ be a prime factor of $d$. Then $p$ does not divide some $f'_i$.
But then $f'_1T_1 + \ldots + f'_nT_n +f'$ is not a unit modulo $p$, and the condition of Corollary \ref{factorialconnected}
holds even for every $p$.

Finally, we assume that $R$ is a principal ideal domain and that
$(d,f_1' ,\ldots ,f_n')=(e)$, where $e\in R$ is not a unit. Let $p\in R$ be a prime element dividing $e$. We still work with the factorization
$h=dh'$, where $h'$ is irreducible and describes the horizontal component.
The elements $ f_1/d,\ldots ,f_n/d, f/d$ do not have a common prime
factor, hence $p$ does not divide $f/d$. Therefore in the field $R/(p)$ the element $f/d$ becomes a unit $u$
and the polynomial $h'$ becomes
$0T_1+ \ldots + 0 T_n + u$. Therefore the horizontal component $V(h')$ and the vertical component $V(p)$ are disjoint and the forcing
algebra is not connected by Corollary \ref{connectedcomponentsone}.
\end{proof}

\begin{example}
\label{connectedexample}
The condition of being a principal ideal domain for $R$ in the last part of Corollary \ref{PID} is necessary,
as the following example shows. With the notation from above we consider the following setting, where
$K$ denotes an arbitrary field: $R:=K[X,Y]$, $B=R[T]$, $h=X^2YT-XY=XY(XT-1)$ and $A:=B/(h)$. Clearly,
$d=\operatorname{gcd}(X^2Y,XY)=XY$, $f_1=X^2Y$ but
$\operatorname{gcd}(d,f_1/d)=\operatorname{gcd}(XY,X)=X\neq1$. Besides,
as Lemma \ref{horizontalvertical} or Lemma \ref{forcingprimedecomposition} shows, the irreducible components of $\Spec A$ are the horizontal component
$V((XT-1)A)$ and the vertical components $V(XA)$ and $V(YA)$. Furthermore,
$V(XA)\cap V(YA) = V((X,Y)A) \neq \emptyset$,
so the two vertical components are connected. Because of
$V(Y A) \cap V((XT-1)A)=V((Y,XT-1)A) \neq \emptyset $
(note also that $V(XA) \cap V((XT-1)A)=V((X,XT-1)A) =  \emptyset $) the condition of
Theorem \ref{connectedcomponents} (or Corollary \ref{factorialconnected}) is fulfilled and hence $\Spec A$ is connected.
However, the condition of a greatest common divisor in Corollary \ref{PID} does not hold.
\end{example}

\begin{example}
\label{behnkekegelexample}
We consider the domain $R=K[X,Y,Z]/(Z^2-XY)$ over a field $K$. This is not a factorial domain, since
$Z^2=XY$
can be written in two ways as a product of irreducible factors.
Accordingly, the rational function $q= \frac{Z}{X} = \frac{Y}{Z}$ is defined
on $D(X,Z)$, and $(X,Z)$ is a prime ideal of height one not given by one element. We look at the forcing
algebra
\[B=R[T]/(XT-Z) \, . \]
The element $XT-Z$ is irreducible in $R[T]$, but not prime. The minimal prime ideals over $(XT-Z)$
are $(XT-Z,Z T-Y)$
(which describes the horizontal component in the spectrum of the forcing algebra $B$, corresponding to the closure
of
the graph of the rational function $q$) and the vertical component $(X,Z)R[T]$. Because of
$(X,Z) + (XT-Z,ZT-Y)=(X,Y,Z)$, these two components intersect and therefore the forcing algebra is connected.
\end{example}

\section{Local properties}

An interesting question is whether the connectedness of $Y=\Spec A$ is a local property over the base $X=\Spec R$.
Specifically, is it true that $Y$ is a connected space if and only if $X$ is connected and for every
$\idealp \in X$, $\Spec A_\idealp$ is  connected, where $A_\idealp $ denotes the localization of $A$ at the
multiplicative system $R\setminus \idealp$ (considered in $A$). The next theorem gives a positive answer to the ``if'' part of this
question in general. For a forcing algebra, the converse holds for a one-dimensional domain, but neither over a reducible curve nor over the affine plane.

\begin{theorem}
\label{theoremlocal}
Let $\psi:R\rightarrow A$ be a ring homomorphism. Set $X:=\Spec R$ and $Y:=\Spec A$. Suppose that $X$ is a
connected space and that for all $\idealp \in X$, $\Spec A_\idealp$ is connected, where
$A_\idealp :=A_{R \setminus \idealp} $, that is, $Y$ is locally (over the base) connected. Then $Y$ is connected.
\end{theorem}
\begin{proof} 
We first show that we can assume that $\psi$ is injective: first, note that  for any minimal prime
$\idealp\in X$, the space $\Spec A_\idealp$ is not empty, because it is connected (our convention is that the empty set is not connected).
Let $Q \in \Spec A_\idealp$ be a prime ideal, then $\psi^{-1}(Q) $ is a prime ideal of $R$ contained in $\idealp$, because
$\psi^{-1}(Q) \cap (R\setminus \idealp)=\emptyset$, moreover it is equal to $\idealp$ in view of the minimality of $\idealp$.
Therefore, for any minimal prime in $R$, there exists a prime ideal $Q$ in $A$ lying over it. Therefore for any
$a\in \operatorname{ker} \psi$, we know that $\psi(a)=0 \in \cap_{Q \in Y}Q$
and then
$a\in \cap_{Q\in Y} \psi^{-1}(Q)   \subseteq \cap_{\idealp \in \operatorname{min} R}\idealp =\operatorname{nil} R$,
that is $\operatorname{ker} \psi \subseteq \operatorname{nil} R$.

In consideration of this it is enough to reduce to the case of $R$ being reduced.
For this reduction consider the natural homomorphism
$\psi_{\operatorname{red}} :=  R_{\operatorname{red}} \rightarrow A_{\operatorname{red}}$ induced by $\psi$,
killing the nilpotent elements. Now, our hypothesis of locally (over the base $X$) connected and the conclusion
holds for $Y$ if and only if if holds (over the base
$X_{\operatorname{red}}=\Spec (A_{\operatorname{red}})$) for
$Y_{\operatorname{red}}:=\Spec (A_{\operatorname{\operatorname{red}}})$. In fact, clearly
$X\approxeq X_{\operatorname{red}}$ and $Y\approxeq Y_{\operatorname{red}}$ as topological spaces, besides, for
any $\idealp \in X_{\operatorname{red}}$,
$(A_{\operatorname{red}})_\idealp\approxeq A_\idealp/(\operatorname{nil} A)A_\idealp$ and $(\operatorname{nil} A)A_\idealp
\subseteq \operatorname{nil}(A_\idealp)$, hence $\Spec ((A_{\operatorname{red}}))_\idealp \approxeq \Spec (A_\idealp/\operatorname{nil} A_\idealp)
\approxeq \Spec A_\idealp$.
In conclusion, it is enough to prove the theorem in the reduced case for injective $\psi$.

Now, we assume that $Y$ is not connected, which is equivalent to say that there exists nontrivial idempotents
$e_1,e_2\in A$ with $e_1+e_2=1$, $e_1e_2=0$ and $e_1,e_2\neq0,1$. Set $J_i=\operatorname{Ann}_R(e_i)$ for $i=1,2$.
We claim that $J_1+J_2\varsubsetneq R$. Otherwise there exists $y_i\in J_i$ such that $y_1+y_2=1$, and then
$y_1y_2=y_1y_2(e_1+e_2)=y_2(y_1e_1)+y_1(y_2e_2)=0+0=0$. Therefore $X=V(y_1) \uplus V(y_2)$, that is, we can write
$X$ as a disjoint union of two closed subsets, which implies in view of the connectedness of $X$ that one of these
closed subsets is empty, or what is the same, one of the $y_i$ is a unit. Hence, $e_i=y_i^{-1}(y_ie_i)=y_i^{-1}0=0$, a contradiction.

So let $J_1+J_2 \subseteq P$ be a prime ideal. By assumption, $A_{R \setminus \idealp}$ is connected, hence either
$e_1$ or $e_2$ become $0$ in this ring.
This means (in the first case) that there exists $s \in R \setminus \idealp$ such that $se_1 =0 $ in $A$. But
then we get the contradiction $s \in J_1$. In conclusion, $Y$ is a connected space.
\end{proof}

We deal next with the one-dimensional case.

\begin{corollary}
\label{onedimlocal}
Suppose that $R$ is a noetherian domain of dimension $1$.
Let $I=(f_1, \ldots , f_n) \neq 0$ be an ideal, $f \in R$ an element and
$A=R[T_1, \ldots ,T_n]/(f_1T_1 + \ldots +f_nT_n+f)$ 
the forcing algebra for these data. Then $\Spec A$ is connected if and only if $\Spec A$ is locally connected,
i.e. for every prime ideal $\idealp \in \Spec R$ is $A_{R \setminus \idealp}$ connected.
\end{corollary}
\begin{proof}
The global property follows from the local property by Theorem \ref{theoremlocal}. So suppose that $\Spec A$ is
connected. By the assumption $I\neq 0$ we know that a horizontal component exists. Hence the fiber over the generic
point $(0)$ is nonempty, thus connected by Lemma \ref{forcingfiber}. The connectedness of $\Spec A$ means by Corollary
\ref{connectedcomponentsone} that the horizontal component meets every vertical component. The vertical components
of $\Spec A_\idealp$ over $\Spec R_\idealp $ for a maximal ideal $\idealp$ in $\Spec R$ are empty or $V(\idealp A)$,
and in the second case they are a vertical component of $\Spec A$. By the intersection condition the horizontal
component and this vertical component (if it exists) intersect, so $\Spec A_\idealp$ is connected.
\end{proof}

The following example shows that for a non-integral one-dimensional base ring, connectedness is not a local
property.

\begin{example}
Let $K$ be a field and let $R=K[X,Y]/(XY(X+Y-1))$. Its spectrum has three line components forming a triangle meeting in $(0,0)$, $(1,0)$ and $(0,1)$.
Consider the forcing algebra
\[ A=R[T]/( (Y+X^2)T - X(X+Y-1) \, .\]
Its spectrum consists in a horizontal line $H_1$ over $X=0$, a horizontal line $H_2$ and one (or two) vertical components
over $X+Y=1$, a vertical line $V$ over $X=Y=0$ and the graph $G$ of the rational function $(X-1)/X$ over $Y=0$.
Because of $G \cap H_2 =\{ (1,0,0) \}$, $H_1 \cap H_2 = (0,1,0)$ and $ H_1 \cap V=(0,0,0) $,
the forcing algebra $A$ is connected. However, the localization of the forcing algebra at $(X,Y)$ is not connected,
because the connecting component $H_2$ is missing (the two connected components are $V \cup H_1$ and $G$).
\end{example}

\begin{corollary}
\label{dedekind}
Suppose that $R$ is a Dedekind domain, i.e. a normal noetherian domain of dimension $1$.
Let $I=(f_1, \ldots , f_n)$ be an ideal, $f \in R$ an element
inside the radical of $I$ and $A=R[T_1 , \ldots ,T_n]/(f_1T_1 + \ldots +f_nT_n+f)$ be
the forcing algebra for these data.
Then the following are equivalent.

\begin{enumerate}
\item 
$\Spec A$ is connected.

\item
$\Spec A$ is locally connected, i.e. for every prime ideal $P \in \Spec R$ is $A_{R \setminus P}$ connected.

\item
$f \in I$.
\end{enumerate}
\end{corollary}
\begin{proof}
The equivalence between (1) and (2) follows from Corollary \ref{onedimlocal}. For the equivalence with (3) we may
assume that $R$ is local, i.e. a discrete valuation domain. Let $p$ be a generator of its maximal ideal. We may
assume at once that $I \neq 0$, because else $f=0$ due to the radical assumption, and also that all $f_i$
(If $f_i=0$ we can omit this without changing any property) and $f$
are not $0$. We write
$f_i= u_i p^{k_i}$ and $f=up^k$ with units $u_i,u$. 
Assume that $f \not\in I$. Then $k < \operatorname{min}(k_1, \ldots, k_n)$. We write the forcing equation as
\[ p^k (u_1p^{k_1-k} T_1 + \ldots +u_np^{k_n-k}T_n+u)\, ,  \]
where the exponents $k_i-k$ are all positive. Because of the radical assumption we have $k \geq 1$. But then the
forcing algebra has the two components $V(p)$ and $V(u_1p^{k_1-k} T_1 + \ldots +u_np^{k_n-k}T_n+u)$ which are
disjoint. The other direction follows from Proposition \ref{connectedprop} (4).
\end{proof}

For a non-normal one-dimensional domain this equivalence can never be true because of Corollary \ref{normalization} below.
The next trivial example shows that this statement is also not true without the radical assumption.

\begin{example}
For $R=K[X]$, $K$ a field, the forcing algebra $K[X,T]/(0T-X) \cong K[T]$ is connected,
but the fiber over the generic point is empty, hence not connected.
\end{example}

In higher dimension, even for a factorial domain, the converse of Theorem \ref{theoremlocal} is also not true.

\begin{example}
\label{connectedexampletwo}
We continue with Example \ref{connectedexample}, i.e. $R:=K[X,Y]$ is the ring of polynomials in two variables over
a field $K$, $B:=R[T]$, $h:=X^2YT-XY$ and $A:=B/(h)$. The morphism $\Spec A \rightarrow \Spec R$ is surjective. We
now already that $\Spec A$ has the three irreducible
components $V(X)$, $V(Y)$ and $V(XT-1)$ that it is connected. However, if we localize $A$ in
$\idealp:=(X)\in \Spec R$,
that is, if we consider the ring $A_\idealp = A_{ R\setminus \idealp  }$, then $\Spec A_\idealp$ has just two
irreducible components, namely, $V((X)A_\idealp)$ and $V((XT+1)A_\idealp)$, because the minimal primes of
$A_\idealp$ are just $(X)A_\idealp$ and $(XT+1)A_\idealp$, since the remaining minimal prime ideal $(Y)$ meets
$R\setminus \idealp$, since $Y \in R\setminus \idealp $. Moreover, these two irreducible components are disjoint,
because
$V((X)A_\idealp))\cap V((XT-1)A_\idealp)=V((X,XT-1)A_\idealp)=V((1))=\emptyset$. In conclusion, $\Spec A_\idealp$
is not connected.
\end{example}

\section{Integral closure and connectedness}
\label{integral}

In this final section we relate the integral closure of an ideal to the universal connectedness of the forcing algebra. Recall that the \emph{integral closure}, written $\overline{I}$, of an ideal $I \subseteq R$ is defined as
\[ \overline{I} := \{f \in R| f^n +a_{1} f^{n-1} + \ldots +a_{n-1} f^1 +a_n =0 \text{ where } a_j \in I^j\} \, .\]
For a noetherian domain, there exists a (discrete) valuative criterion for the integral closure: The containment $f \in \overline{I}$ holds if and only
if for all ring homomorphisms $\theta:R \rightarrow D$ to a discrete valuation domain $D$
we have $\theta(f) \in ID$, see \cite[Theorem 6.8.3]{hunekeswanson}.

\begin{definition}
Let $\varphi:Y\rightarrow X$ be a morphism between affine schemes. We say that $\varphi$ is a universally connected if $W\times_X Y$ is 
connected for any affine noetherian change of base $W\rightarrow X$, with $W$ connected.
\end{definition}

Now, we prove a criterion for belonging to the integral closure in terms of the universal connectedness of the corresponding forcing morphism.
\begin{theorem}
Let $A$ be a forcing algebra over a noetherian ring $R$ and $\varphi:Y:={\rm Spec}\, A\rightarrow X:={\rm Spec}\, R$ the corresponding forcing morphism. Then the 
following conditions are equivalent:
\begin{enumerate}
\item $f$ belongs to the integral closure of $I$, i.e. $f\in \overline{I}$.
\item $\varphi$ is a universal submersion.
\item $\varphi$ is universally connected.
\item $W\times_X Y$  is connected for all change of base of the form $W={\rm Spec}\,D$, where $D$ is a discrete valuation domain.
\end{enumerate}

\end{theorem}
\begin{proof}
$(1)\Rightarrow(2)$. Recall that a submersion is universal if it remains a submersion under noetherian change of base.
 Due to the valuative criterion for the integral closure  (see \cite[Theorem 6.8.3]{hunekeswanson})
and the fact that (2) can be checked after change of base to a discrete valuation domain $D$ (see \cite[Remarque 2.6]{SGA1}),
we can assume that $R=D$ and that $f\in I$ and we have to prove that $\varphi:{\rm Spec}\, A \rightarrow {\rm Spec}\, D$ is a 
submersion. But $f\in I$ if and only if there exists a section $s:{\rm Spec}\,D\rightarrow {\rm Spec}\,A$, i.e. 
$\varphi\circ s=\operatorname{Id}_{{\rm Spec}\,D}$. By an elementary topological argument the existence of a section implies that
$\varphi$ is a submersion.

$(2)\Rightarrow (3)$. Let $W\rightarrow X$ be an affine noetherian connected change of base, then since $\varphi_W:W\times_XY\rightarrow W$ is 
a submersion, by Proposition \ref{connectedprop} (5) $W\times_X Y$ is connected.

$(3)\Rightarrow(4)$ is trivial.

$(4)\Rightarrow (1)$. Let $ W={\rm Spec}\, D  \rightarrow X  $ be a change of base, where $D$ is a 
discrete valuation domain. By the valuative criterion for integral closure it is enough to show that $f\in ID$. First, note that $f\in {\rm rad}(ID)$, 
which is equivalent to say that $\varphi_W: W\times_X Y \rightarrow W$ is surjective (see \cite{brennerforcingalgebra}).
If the fiber over $\fop \in W$ were empty,
then the morphism
\[ \Spec \kappa(\fop)[[x]] \longrightarrow  \Spec \kappa(\fop) \longrightarrow W \longrightarrow X\]
would yield a contradiction, since then the pull-back $ (\Spec \kappa(\fop)[[x]] )  \times_XY $ would be empty (hence not connected)
and $\kappa(\fop)[[x]]$ is a discrete valuation domain. Thus $f\in {\rm rad}(ID)$ and by Corollary \ref{dedekind} the
connectedness of $W\times_XY$ is equivalent to $f\in ID$.
\end{proof}

Now, we prove a corollary of this theorem charaterizing the property that a fraction belongs to the integral closure (or normalization) of an 
integral domain. 
\begin{corollary}
\label{normalization}
Let $R$ be an integral domain, $K=Q(R)$ its field of fractions. Let $r/s\in K$ with $s\neq 0$, let $A=R[T]/(sT+r)$ be the forcing algebra and 
$\varphi:Y:={\rm Spec}\, A\rightarrow X:={\rm Spec}\, R$ the corresponding morphism. Then $r/s$ is integral over $R$ if and only if $\varphi$
is universally connected.
\end{corollary}
\begin{proof}
It is an elementary fact that $r/s$ is integral over $R$ if and only if $r\in \overline{(s)}$. Then by the former theorem 
$r\in \overline{(s)}$ if and only if $\varphi$ is universally connected.
\end{proof}

In our final example we show that a forcing algebra over a non-normal curve might be connected but not universally connected.
In fact the pull-back to the normalization is already not connected.

\begin{example}
Let $K$ be a field and consider the ring-homomorphism $K[u,v]  \rightarrow K[x],\, u \mapsto x(x-1),\, v \mapsto x^2(x-1)$.
The kernel of this is $(u^3+uv-v^2)$. Let $R= K[u,v]/(u^3+uv-v^2)$. Because of $v/u=x$ the induced homomorphism
$ K[u,v]/(u^3+uv-v^2) \rightarrow K[x] $ is the normalization. We consider the forcing algebra $A=R[T]/(vT+u)$.
It consists of a horizontal component given by $V(vT+u,vT^3 +T+1)$ (check that this is a prime ideal) and the
vertical component $V(u,v)$. They intersect in $V(u,v,T+1)$, hence the forcing algebra is connected.
When we pull-back this situation to the normalization
we get
\[A' =K[x][T]/( x^2(x-1)T + x(x-1) ) \cong K[x][t]/(x(x-1) ( xT + 1) ) \, .\]
Now we have one horizontal component and two vertical components, and the horizontal hyperbola meets exactly one of them,
hence this forcing algebra is not connected by Corollary \ref{connectedcomponentsone}.
\end{example}

\bibliographystyle{amsplain}
\bibliography{bibliothek}

\end{document}